\documentclass[12pt,twoside,reqno]{amsart}

\usepackage{hyperref}
\usepackage{amsthm, amsmath, amscd, amssymb,centernot}
\usepackage[all]{xy}
\usepackage[T1]{fontenc}
\usepackage[left=2.5cm,top=2.5cm,bottom=3cm,right=2.5cm]{geometry}
\usepackage{amsthm, amsmath, amscd, amssymb,centernot}
\usepackage{comment}
\usepackage{tikz}

\usepackage{graphicx}
\usepackage{booktabs}
\usepackage{caption}
\usepackage{float}

\newcommand{\mult}{\text{mult}}
\newcommand{\s}{\varepsilon}

\newcommand{\str}{\mathcal{O}}
\newcommand{\proj}{\mathbb{P}}

\numberwithin{equation}{section}
\newtheorem{theorem}{Theorem}[section]
\newtheorem*{theorem*}{Theorem}
\newtheorem{proposition}[theorem]{Proposition}
\newtheorem{lemma}[theorem]{Lemma}
\newtheorem{corollary}[theorem]{Corollary}

\theoremstyle{definition}
\newtheorem{question}[theorem]{Question}
\newtheorem{remark}[theorem]{Remark}
\newtheorem{example}[theorem]{Example}

\input xy
\xyoption{all}

\begin{document}
\title{Positivity on Blow-up of hyperelliptic surfaces}
\author[Praveen Roy]{Praveen Kumar Roy}
\address{UM-DAE Centre for Excellence in Basic Sciences, University of Mumbai Santacruz, Mumbai 400098, India}
\email{praveen.roy@cbs.ac.in}


\maketitle

\begin{abstract}
Let $X_r$ denote the blow-up of the hyperelliptic surface $X$ at $r$ very general points. In this paper, we first provide a criterion for the ampleness of a line 
bundle on $X_r$ and compare it with an existing result. We then study the multi-point Seshadri constants of ample line bundles on hyperelliptic surfaces $X$. 
Next, we compute single-point Seshadri constants on $X_r$ for specific ample line bundles on odd types. 
Furthermore, we show that the global Seshadri constants for certain ample line bundles on blow-up of hyperelliptic surfaces are rational.
\end{abstract}

\section{Introduction}
Let $X$ be a smooth projective algebraic variety and let $L$ be an ample line bundle on $X$. 
To quantify the positivity of $L$ locally near a point $x \in X$, Demailly defined the \textit{Seshadri constant of $L$ at a point $x$}, denoted $\s(X, L, x)$ \cite{Dem}. 

For a positive integer $r > 0$, a natural generalization of this is the \textit{multi-point Seshadri constant of $L$} at the points $x_1, x_2, \ldots, x_r \in X$, defined as follows:
\[
\s(X, L, x_1, x_2, \ldots, x_r) :=  \inf\limits_{\substack{C \subset X {\rm \; is\; a\; curve\; with} \\ C \cap \{ x_1, x_2, \ldots, x_r \} \neq \emptyset}} \frac{L\cdot C}{\sum\limits_{i=1}^r \mult_{x_i} C},
\]
where the infimum is taken over reduced, irreducible curves $C$ passing through at least one of the points $x_1, x_2, \ldots, x_r  \in X$. 

Equivalently, this constant can be expressed as: 
\[
\s(X, L, x_1, x_2, \ldots, x_r)  = \sup\{d\; |\; \pi_r^*L - d\sum\limits_{i=1}^r E_i {\rm \; is\; nef} \},
\]
where $\pi_r : X_r \longrightarrow X$ denotes the blow-up of $X$ at the points $x_1, x_2, \ldots, x_r \in X$, and $E_i$ are the corresponding exceptional divisors for $i = 1, 2, \ldots, r$.

Some well-known properties of Seshadri constants on smooth projective surfaces include the following: 

Let $X$ be a smooth projective surface, 
and let $L$ be an ample line bundle on $X$. For $r >0$ points $x_1, x_2, \ldots, x_r  \in X$, 
we have $$0 < \s(X, L, x_1, x_2, \ldots, x_r) \leq \sqrt{\frac{L^2}{r}}.$$ 
Moreover, for $r$ very general points, $\varepsilon(X, L, x_1, x_2, \ldots, x_r )$ remains constant and is denoted as $\varepsilon(X, L, r)$. 

When $r=1$, we recover the single point Seshadri constant $\s(X, L, x)$ of $L$ at $x \in X$. Using this definition, the Seshadri criterion for ampleness of a line bundle $L$ 
on a surface $X$ \cite[Theorem 7.1]{Har70} states that $L$ is ample if and only if 
\[
\s(X, L, x) > 0 \quad \text{for all} \;x \in X. 
\]
By taking the infimum of $\s(X, L, x)$ over all $x \in X$, 
we obtain the \textit{global Seshadri constant of $L$}:
\[
\s(X, L) := \inf\limits_{x\in X}\{\s(X, L, x)\}.
\]
For an ample line bundle $L$, these constants, $\s(X, L, 1)$, $\s(X, L, x)$, and $\s(X, L)$ are positive real numbers and satisfy the following inequalities: 
\[
0 < \s(X, L) \leq \s(X, L, x) \leq \s(X, L, 1) \leq \sqrt{L^2}.
\]
However, computing them precisely or determining whether they are rational numbers, remains a challenging and active area of research.

Seshadri constants on surfaces have been extensively studied, and a list of related works on single-point Seshadri constants can be found in \cite{Bau, BS, primer, BS1, EL, FSST, Fue,  Ogu, Roy, Syz, S}, 
while those on multi-point Seshadri constants are available in \cite{TB, EL, FSST, Han, HRS, Harb, Harb-Jo, Roy, SSyz, Szemb-1, Szemb-2}.

In this article we first obtain a criterion for the ampleness of a line bundle on the blow-up of hyperelliptic surfaces and then compute the single point Seshadri constants of certain ample line bundles on them. 
For hyperelliptic surfaces themselves, results on single point Seshadri constants have already been obtained (see \cite{Far, HR}). We further study multi-point Seshadri constants of ample line bundles on hyperelliptic surfaces. 


A \textit{hyperelliptic surface} $X$ is a minimal smooth surface with Kodaira dimension $\kappa(X)=0$ such that $h^1(X,\str_X) = 1$ and $h^2(X,\str_X) = 0$. 
Alternatively, a hyperelliptic surface is defined as a smooth surface $X$ such that $X \cong (A\times B)/G$, where $A$ and $B$ are elliptic curves, and $G$ 
is a finite group of translation of $A$ acting on $B$ in such a way that $B/G \cong \proj^1$. 

We have the following diagram:

\[
\xymatrix{ & X \cong  (A\times B)/G \ar[dl]_{\Psi}\ar[dr]^-{\Phi} \\ 
B/G \cong \proj^1 & &  A/G}
\]

Here, $\Phi$ and $\Psi$ denote the natural projections. 
The fibres of $\Phi$ are all smooth and isomorphic to $B$, whereas the fibres of $\Psi$ are multiples of smooth elliptic curves, and
all but finitely many of them are smooth and isomorphic to $A$. 

For a smooth surface $X$, let $\text{Num}(X)$ denote the group of divisors modulo numerical equivalence classes and is defined as $\text{Div}(X)/\thicksim$, 
where $`\thicksim'$ denote the numerical equivalence of divisor. Hyperelliptic surfaces has been classified into seven types long ago \cite{Bea}. In fact, the basis of 
$\text{Num}(X)$ over $\mathbb{Z}$ has also been computed by Serrano in \cite[Theorem 1.4]{Se}. We include the table for the convenience of the reader. 

\begin{theorem}[Serrano]
Let $X \cong  (A\times B)/G$ be a hyperelliptic surface with two natural projections $\Phi$ and $\Psi$. Then the generators for ${\rm Num} (X)$ and the 
multiplicities $m_1,\ldots, m_s$, where each $m_i$ represents the multiplicity of a singular fibre of $\Psi$ and $s$ denote the number of such fibres, 
are as in the table below:
\begin{center} 

  \begin{tabular}{|c| c| c| c|}
 \hline
 
 Type of $X$  & $G$ &$ m_1,m_2,
\ldots,m_s$ & Basis of Num($X$)\\
 \hline 
 1 & $\mathbb{Z}_{2}$ &$2,2,2,2$&$A/2$, $B$ \\ 
 2 & $\mathbb{Z}_{2} \times \mathbb{Z}_{2}$&$2,2,2,2$&$A/2$, $B/2$ \\
 3 & $\mathbb{Z}_{4}$ &$2,4,4$&$A/4$, $B$ \\
 4 & $\mathbb{Z}_{4} \times \mathbb{Z}_{2}$ & $2,4,4$ &$A/4$,$B/2$  \\
 5 & $\mathbb{Z}_{3}$ & $3,3,3$ & $A/3, B$ \\
 6 & $\mathbb{Z}_{3} \times \mathbb{Z}_{3}$ & $3,3,3$ & $A/3,B/3$ \\
 7 & $\mathbb{Z}_{6}$ & $2,3,6$ & $A/6,B$ \\
\hline 
  
 \end{tabular}

\end{center}
\end{theorem}

\subsection{Notation and convention:} 
Following Serrano's notation in \cite{Se}, we denote by $\mu$ to be the least common multiple of $\{m_1, m_2, ..., m_s \}$. Let $\gamma$ represent the order of the group $G$. 
Since $\text{Num}(X)$ is generated by $A/\mu$ and $\mu B/\gamma$, an ample line bundle $L$ on $X$ can be expressed as $L = a \frac{A}{\mu} + b \frac{\mu B}{\gamma}$, 
where $a, b > 0$ (\cite[Lemma  1.3]{Se}). For brevity, we denote this line bundle as $L \equiv (a, b)$. The following properties hold for line bundles on $X$: 
\begin{enumerate}
\item $A^2 = 0 = B^2$. 
\item $A\cdot B = \gamma = |G|$.
\item \text{A line bundle $(0,b) \equiv b\frac{\mu}{\gamma}B$ is effective if and only if $ b\frac{\mu}{\gamma} \in \mathbb{N}$} (\cite[Proposition 5.2]{Apr}).
\end{enumerate}

Let $X$ be a hyperelliptic surface and let $x_1, \ldots, x_r \in X$ be $r$ very general points. Let $\pi_r : X_r \longrightarrow X$ 
denote the blow-up of $X$ along these $r$ points with exceptional divisors $E_i$ such that $\pi_r(E_i) = x_i$ for $i = 1,2,\dots,r$. 

Let $L'$ denote the pull back of a line bundle $(a, b)$ on $X_r$, and let $d_1, \ldots, d_r$ be $r \geq 0$ are non-negative integers. Then, a line bundle on $X_r$ can be expressed as 
\[
L =  L' - \sum\limits_{i = 1}^r d_i E_i, 
\]
which we denote as $L \equiv (a, b, d_1, d_2, \ldots, d_r)$. In the special case when $d_1 = d_2 = \cdots = d_r =  d$, we denote $L \equiv (a, b, \textbf{d})$. 
Additionally, if $d_j = 0$ for all $j \neq i$ for some $i$, and $d_i = k$, we denote such a line bundle by $L \equiv (a, b, k_i)$. 

Let $L \equiv (a, b, \textbf{d})$ be a line bundle on $X_r$. By abuse of notation, we will continue to write $L \equiv (a, b, \textbf{d})$ and interpret this as 
a line bundle with $d_i = d$ for all $i = 1,2,\ldots, r+s$ on $X_{r+s}$, for all values of $s\ge0$. Since we are working on a surface, we will follow the same convention for curves as well. 
 
The strict transform under $\pi_r$ of the fibre of $\Psi$ and $\Phi$ containing the blown-up point $x_i$ will be denote as $A-E_i \equiv (\mu, 0, 1_i)$ and $B-E_i  \equiv (0, \frac{\gamma}{\mu}, 1_i)$, 
respectively. We will refer to the smooth fibres of $\Psi$ as $`\textbf{Smooth } A$' and the singular fibres with the highest multiplicity 
(e.g., fibres of multiplicity $2, 2, 4, 4, 3, 3,$ and $6$ on surfaces of types $1, 2, 3, 4, 5, 6$, and $7$, respectively) as $`\textbf{Singular } A$' on $X_r$. 

The paper is organized as follows: In Section \ref{ampleness on the blow-up of hyperelliptic surfaces}, we begin by providing a criterion for the ampleness of line bundles on the blow-up 
of hyperelliptic surfaces, as presented in Theorem \ref{ampleness:theorem}. We then compare this result with an existing one in Remark \ref{Comparison}. In Section \ref{Seshadri constant section}, 
we begin by exploring multi-point Seshadri constants. For certain special points $x_1, x_2, \ldots, x_r \in X$, we compute these constants on hyperelliptic surfaces of odd type in Proposition \ref{prop:multi-point-odd-type} 
and on surfaces of type 2 and type 4 in Propositions \ref{prop:multi-point-even-type-2} and \ref{prop:multi-point-even-type-4}, respectively. For hyperelliptic surfaces of type 6, we establish a bound in Proposition \ref{prop:multi-point-even-type-6}. 

In the latter part of Section \ref{Seshadri constant section}, we compute the single point Seshadri constants of some ample line bundles on the blow-up $X_r$ of hyperelliptic surfaces. More specifically, 
we provide a bound on these constants in Theorem \ref{Thm-type1-general-point} and compute 
their exact values under certain conditions in Corollary \ref{Cor:general-point}. Additionally, we show that the global Seshadri constant of certain ample line bundles is a rational 
number on the blow-up $X_r$ of hyperelliptic surfaces of odd type in Theorem \ref{Global:SC:Odd-Type}, of type $2$ and $4$ in Theorem \ref{Global:SC:Even-Type-2&4}, and of type $6$ in Theorem \ref{Global:SC:Even-Type-6}.

\section{Ampleness of line bundle on blow-up of Hyperelliptic surfaces}\label{ampleness on the blow-up of hyperelliptic surfaces}

We begin by recalling the following ampleness criterion proved in \cite{Kuchle}. 

\begin{theorem}[Corollary  \cite{Kuchle}]\label{Kuchle's:theorem}
	Let $S$ be a smooth complex projective surface and $H$ be an ample divisor on $S$. Let $S'$ denote the blow-up of $S$ along $r$ general 
	points with exceptional divisors $E_i$ and let $H'$ denote the pull back of $H$ on $S'$. Let $n \geq 3$, then the line bundle $L =  nH'-E_1-E_2-\cdots-E_r$ is ample if and only if $L^2 > 0$.
\end{theorem}

\begin{remark}
	As mentioned in the remark of \cite{Kuchle}, the above theorem also holds for $n\geq 2$, provided that any two general points on $S$ can not be joined by a curve 
	$C\subset S$ such that $H \cdot C =1$. In particular, this condition is satisfied when $H^2 \geq 2$, which holds for all ample line bundles on hyperelliptic surfaces.
\end{remark}

The following lemma essentially follows from Theorem \ref{Kuchle's:theorem}.

\begin{lemma}\label{ampleness:lemma}
Let $X$ be a hyperelliptic surface, and let $\pi_r : X_r \longrightarrow X$ denote the blow-up of $X$ at $r$ general points with exceptional divisors $E_i$ for $i = 1, 2, \ldots, r$.
Let $L \equiv (a, b, \textbf{d})$ be a line bundle on $X_r$. If $ \lfloor \frac{a}{d} \rfloor,  \lfloor \frac{b}{d} \rfloor > \max\{1, \sqrt{\frac{r}{2}} \}$, then $L$ is ample.
\end{lemma}

\begin{proof}
Note that $L\cdot C > 0$ if and only if $(a/d, b/d, \textbf{1}) \cdot C > 0$ for all curves $C \subset X_r$. The latter inequality holds if $(\lfloor a/d \rfloor, \lfloor b/d \rfloor, \textbf{1}) \cdot C> 0$ for all curves $C \subset X_r$. 
Therefore, the ampleness of $L$ follows from the ampleness of $M := (\lfloor a/d \rfloor, \lfloor b/d \rfloor, \textbf{1})$, which in turn follows from Theorem \ref{Kuchle's:theorem} and the above remark. Since $\lfloor a/d \rfloor, \lfloor b/d \rfloor \geq 2$ and 
\begin{eqnarray*}
M^2 &=& 2 \lfloor \frac{a}{d} \rfloor \lfloor \frac{b}{d} \rfloor - r \\
&>&  2 \sqrt{\frac{r}{2}} \cdot \sqrt{\frac{r}{2}} - r \\
&=& 0.
\end{eqnarray*}
\end{proof}

For a line bundle $L$ as in the above lemma, note that if $L^2 > 0$, then Theorem \ref{Kuchle's:theorem} implies that $L$ is ample provided that $\lfloor \frac{a}{d} \rfloor , \lfloor \frac{b}{d} \rfloor \geq 2$. 
Therefore, the bounds on $\lfloor \frac{a}{d} \rfloor$ and $\lfloor \frac{b}{d}\rfloor$ given in the above lemma are sufficient condition for $L$ to be ample, but they are not necessary. 
Furthermore, if $L^2 > 0$, then at least one of $\frac{a}{d}$ or $\frac{b}{d}$ must be strictly greater than $\sqrt{\frac{r}{2}}$. 

We now present the main result of this section.
\begin{theorem}\label{ampleness:theorem}
	Let $X_r$ be the blow-up of a hyperelliptic surface $X$ along $r$ general points with exceptional divisors $E_i$ for $i = 1, 2, \ldots, r$. Let $d_1, d_2, \ldots, d_r$ be non-negative 
	integers, and let $L \equiv (a, b, d_1, d_2, \ldots, d_r)$ be a line bundle on $X_r$. 
	Then $L$ is ample if the following conditions hold
	\begin{enumerate}
		\item $a, \mu b > d_i$ for all $i$,
		\item $a+b > \mu \sum\limits_{i=1}^{r} d_i$.
	\end{enumerate}
\end{theorem} 
\begin{proof}
Let $L \equiv (a, b, d_1, d_2, \ldots, d_r)$ be a line bundle on $X_r$ satisfying the conditions (1) and (2). We will prove the ampleness of $L$ by showing that 
$L^2 >0$ and $L \cdot C > 0$ for all reduced and irreducible curve $C \subset X_r$. 

Note that the strict transforms of fibres of $\Phi$ and $\Psi$ that do not contain the blown-up points intersect with $L$ to give:
\[
L \cdot (\mu, 0) = \mu b > 0, \; \text{and} \;\; L \cdot (0,\frac{\gamma}{\mu}) = \frac{\gamma}{\mu} a >0.
\]
On the other hand, the fibres containing the blown-up points intersect with $L$ to give:
\begin{eqnarray}\label{Thm-Section-1.1}
L \cdot (\mu, 0, 1_i) = \mu b - d_i > 0,  \; \text{and} \;\;  L \cdot (0,\frac{\gamma}{\mu}, 1_i) = \frac{\gamma}{\mu} a - d_i >0.
\end{eqnarray}
Now, let $C \equiv (\alpha, \beta, \delta_1, \ldots, \delta_r)$ ($\alpha, \beta \neq 0$) be a reduced and irreducible curve that is different from the types of curve 
considered above. Since $C$ is neither a strict transform of any fibre nor an exceptional divisor (as $L\cdot E_i = d_i > 0$), $\pi_r(C)$ intersects 
positively with each fibre of $\Phi$ and $\Psi$. Therefore, we have 
\[
C \cdot (\mu, 0, 1_i) \ge 0,  \; \text{and} \;\; C \cdot (0, \frac{\gamma}{\mu}, 1_i) \ge 0, 
\]
which gives $\mu \beta \ge \delta_i$ and $\alpha \frac{\gamma}{\mu} \ge \delta_i$ for all $i$. \\
Now, we show that $L \cdot C = a\beta + b\alpha - \sum\limits_{i=1}^r  \delta_i d_i > 0$. To this end, first assume that $\alpha \le \beta$. Then, considering the 
assumption in (2) and multiplying it by $\alpha$, we get
\begin{eqnarray*}
\beta a + \alpha b \ge \alpha a + \alpha b > \mu \alpha \sum\limits_{i=1}^r d_i \ge \alpha \frac{\gamma}{\mu} \sum\limits_{i=1}^r d_i \geq  \sum\limits_{i=1}^r \delta_i d_i,
\end{eqnarray*}
since $\mu^2 \geq \gamma$. 

Therefore, we obtain
\[
L \cdot C = \beta a + \alpha b -  \sum\limits_{i=1}^r \delta_i d_i > 0.
\]
On the other hand, if $\alpha > \beta$, we multiply the inequality in (2) by $\beta$ to obtain
\begin{eqnarray*}
\beta a + \alpha b > \beta a + \beta b > \mu \beta \sum\limits_{i=1}^r d_i  \geq \sum\limits_{i=1}^r \delta_i d_i.
\end{eqnarray*}
Thus, we have 
\[
L \cdot C = \beta a + \alpha b -  \sum\limits_{i=1}^r \delta_i d_i > 0.
\]
Now, it only remains to show that $L^2 = 2ab - \sum\limits_{i=1}^r  d_i^2  > 0$. When $a \le b$, this can be observed by multiplying the equation in (2) by $a$ and using the inequality \ref{Thm-Section-1.1}, as follows:
\begin{eqnarray*}
2ab > a^2 + ab > \mu a \sum\limits_{i=1}^r  d_i \ge \frac{\gamma}{\mu} a \sum\limits_{i=1}^r  d_i \ge \sum\limits_{i=1}^r  d_i^2.
\end{eqnarray*}
In the case when $a > b$, multiplying the equation in (2) by $b$ and using the inequality \ref{Thm-Section-1.1} we obtain $L^2 > 0$ as follows:
\[
2ab > ab + b^2 > \mu b \sum\limits_{i=1}^r  d_i \ge \sum\limits_{i=1}^r  d_i^2.
\]
This completes the proof of the theorem.
\end{proof}

\begin{remark}\label{Comparison}
Note that the K\"{u}chle's theorem (Theorem \ref{Kuchle's:theorem}) when applied to the blow-up $X_r$ of a hyperelliptic surfaces $X$ along $r$-general points, provides 
an ampleness criterion for a line bundle $L\equiv (a, b, {\bf d})$ on $X_r$. However, it says nothing about the line bundle of non-homogeneous type i.e., $L \equiv (a, b, d_1, d_2, \ldots, d_r)$. 

\end{remark}

\begin{remark}
In Theorem \ref{ampleness:theorem}, note that the converse statement does not hold; that is, the ampleness of $L$ does imply (1) but it does not necessarily imply (2). 
This can be illustrated by the following example: Consider $X_{10}$ to be the blow-up of a hyperelliptic surface of type 1 at 10 very general points, and let $L \equiv (3,3, \textbf{1})$ be a line 
bundle on $X_{10}$. It follows from Lemma \ref{ampleness:lemma} that $L$ is ample, but $3 + 3 = 6 \not\ge 2\times 10$. 
\end{remark}

\section{Seshadri constants}\label{Seshadri constant section}

\subsection{Multi-point Seshadri constants on hyperelliptic surfaces}\label{multi-point-S.C} 
\begin{lemma}
Let $X$ be a hyperelliptic surface and and let $L \equiv (a,b)$ with $a , b >1$ be an ample line bundle on $X$. Let $x_1, x_2, \ldots, x_r$ be $r \geq8$ general points of $X$, then 
\[
\s(X, L, r) = \s(X, L, x_1, x_2, \ldots, x_r) \geq \min\left\{ \frac{a}{\lceil \sqrt{r/2} \rceil}, \frac{b}{\lceil \sqrt{r/2} \rceil} \right\}.
\]  
\end{lemma}
\begin{proof}
The lemma follows by observing that a line bundle $L \equiv (a,b, \textbf{d})$ on $X_r$ is nef for $d = \min\left\{ \frac{a}{\lceil \sqrt{r/2} \rceil}, \frac{b}{\lceil \sqrt{r/2} \rceil} \right\}$. 
This, in particular, follows from Lemma \ref{ampleness:lemma}, 
since for $r\geq 8$, $\max\{2, \lceil \sqrt{\frac{r}{2}}\rceil \} = \lceil \sqrt{\frac{r}{2}} \rceil$ and in that case nefness of $L$ follows provided $\lfloor \frac{a}{d} \rfloor , \lfloor \frac{b}{d} \rfloor \geq \sqrt{r/2} $, 
which holds for $d = \min\left\{ \frac{a}{\lceil \sqrt{r/2} \rceil}, \frac{b}{\lceil \sqrt{r/2} \rceil} \right\}$.
\end{proof}

Note that if $\sqrt{r/2}$ is an integer, then the above lower bound of multi-point Seshadri constant converges to the conjectured bound $\sqrt{\frac{L^2}{r}}$ as $|a-b| \mapsto 0$. 
\begin{lemma}
Let $X$ be a hyperelliptic surface and and let $L \equiv (a,a)$ with $a>1$ be an ample line bundle on $X$. Let $x_1, x_2, \ldots, x_r$ be $r \geq8$ 
general points of $X$ such that $\frac{r}{2}$ is a perfect square. Then 
\[
\s(X, L, r) = \s(X, L, x_1, x_2, \ldots, x_r) = \sqrt{\frac{L^2}{r}}.
\]  
\end{lemma}
\begin{proof}
For any ample line bundle $L$, the upper bound for multi-point Seshadri constant at $r$ points i.e., $\sqrt{\frac{L^2}{r}}$, is well known. 
We show that it is also the lower bound under the given assumption.

Note that, 
\[
\s(X, L, x_1, x_2, \ldots, x_r) := \sup\{d \; |\;  \pi_r^*(L) - d E \;{\rm is\; nef}\} = \sup\{d \; |\; L= (a,a, \textbf{d})  \;{\rm is\; nef}\}.
\] 
Following Lemma \ref{ampleness:lemma}, $L$ is nef if $\lfloor \frac{a}{d} \rfloor \geq  \max\{2,  \sqrt{\frac{r}{2}}  \} =  \sqrt{\frac{r}{2}} \Leftrightarrow \frac{a}{d}  \geq \lceil \sqrt{\frac{r}{2}} \rceil = \sqrt{\frac{r}{2}}$ (since $r \geq 8$ and $r/2$ is perfect square), which is equivalent to 
\[
d \leq \sqrt{\frac{2a^2}{r}} = \sqrt{\frac{L^2}{r}}.
\]
Thus, $L$ is nef 
for $d = \sqrt{\frac{L^2}{r}}$ and hence $\s(X, L, r) = \sqrt{\frac{L^2}{r}}.$
\end{proof}

Note that in the above lemma, $\sqrt{\frac{L^2}{r}}$ is a rational number.

We now compute the multi-point Seshadri constant of an ample line bundle $L$ on a hyperelliptic surface at some special points. Let $x_1,x_2,\ldots, x_r\in X$ be $r$ points on $X$. We set the following notation:
\begin{center}
$s_0 := \max\{k \; | \; x_{j_1}, x_{j_2}, \ldots, x_{j_k} \in {\rm Fibre}\; A \}$ \\

$t_0 := \max\{k \; | \; x_{j_1}, x_{j_2}, \ldots, x_{j_k} \in {\rm Fibre}\; B \}$\\

$l_A := \min\{k \; | \; x_1,x_2,\ldots, x_r \in \cup_{i=1}^k ({\rm Fibre}\; A) \}$ \\

$l_B := \min\{k \; | \; x_1,x_2,\ldots, x_r \in \cup_{i=1}^k ({\rm Fibre}\; B) \}$.
\end{center}
Here by `Fibre $A$' and `Fibre $B$' we mean the fibres of $\Psi$ and $\Phi$ respectively. Note that $l_B \geq s_0$ and $l_A \geq t_0$.

\begin{proposition}\label{prop:multi-point-odd-type}
Let $X$ be a hyperelliptic surface of odd type and let $L \equiv (a,b)$ be an ample line bundle on $X$. Let $x_1, x_2, \ldots, x_r$ be $r$ points on $X$ such that $s_0$ of them are on a {\bf Singular} $A$, $l_A = t_0$ and $l_B = s_0$.  
Then
\[
\s(X, L, x_1,x_2,\ldots, x_r) = \min\left\{\frac{a}{t_0} , \frac{b}{s_0}\right\}.
\]
\end{proposition}
\begin{proof}
The Seshadri ratio corresponding to the fibres of $\Phi$ and $\Psi$ containing $t_0$ and $s_0$ points are $a/t_0$ and $b/s_0$, respectively. Therefore, we have
\[
\s(X, L, x_1,x_2,\ldots, x_r) \leq \min\left\{\frac{a}{t_0} , \frac{b}{s_0}\right\}.
\]
To show the equality, let $C \equiv (\alpha, \beta)$ be a reduced and irreducible curve (different from the fibres of $\Phi$ and $\Psi$) passing through $x_1,x_2, \ldots, x_r$ with 
multiplicities ${\rm mult}_{x_1}C, {\rm mult}_{x_2}C, \ldots, {\rm mult}_{x_r}C$, respectively. By B\'ezout's theorem, we get:
\begin{eqnarray*}
C \cdot (\mu, 0) \geq \sum\limits_{i=1}^{k} {\rm mult}_{x_{j_i}}C  \Rightarrow \mu \beta l_A \geq  \sum\limits_{i=1}^{r} {\rm mult}_{x_i} C \\
 {\rm and}  \quad C \cdot (0, 1) \geq  \sum\limits_{i=1}^{k} {\rm mult}_{x_{j_i}}C \Rightarrow  \alpha l_B \geq \sum\limits_{i=1}^{r} {\rm mult}_{x_i} C.
\end{eqnarray*}
Since $l_A = t_0$ and $l_B = s_0$, we get 
\begin{eqnarray}\label{eqn:multi-point-1}
\mu \beta t_0 \geq \sum\limits_{i=1}^r {\rm mult}_{x_i}C, \quad {\rm and} \quad  s_0 \alpha \geq \sum\limits_{i=1}^r {\rm mult}_{x_i}C.
\end{eqnarray}
Now, we note that
\begin{eqnarray}\nonumber
\frac{L \cdot C}{\sum\limits_{i=1}^r {\rm mult}_{x_i}C } &=& \frac{a\beta + b\alpha}{\sum\limits_{i=1}^r {\rm mult}_{x_i}C } \\ \label{inequality-MP-SC}
&\geq& \frac{a}{\mu t_0} + \frac{b}{s_0} \quad {\rm (Using \; equation \ref{eqn:multi-point-1})}\\ \nonumber
&\geq& \min\left\{\frac{a}{t_0} , \frac{b}{s_0}\right\}.
\end{eqnarray}
Therefore, we have
\begin{eqnarray*}
\s(X, L, x_1,x_2,\ldots, x_r) = \min\left\{\frac{a}{t_0} , \frac{b}{s_0}\right\}.
\end{eqnarray*}
\end{proof}

\begin{corollary}
Let $X$ be a hyperelliptic surface of odd type and let $L \equiv (a,b)$ be an ample line bundle on $X$. Let $\mathcal{C}$ be an arrangement of curves formed by the fibres of $\Psi$, namely $A_1, A_2, \ldots, A_p$ such 
that at least one of them is a singular fibre, along with fibres of $\Phi$, denoted by $B_1, B_2, \ldots, B_q$. Let $x_1, x_2, \ldots, x_r$ are the singular points of this arrangement. Then 
\[
\s(X, L, x_1,x_2,\ldots, x_r) = \min\left\{\frac{a}{p} , \frac{b}{q}\right\}.
\]
\end{corollary}
\begin{proof}
Note that $l_A = t_0 = p$ and $l_B = s_0 = q$, therefore it follows from the Theorem \ref{prop:multi-point-odd-type}.
\end{proof}

\begin{remark}
On hyperelliptic surfaces of type 1, the assumption about $l_B = s_0$ can be relaxed to $l_B \leq 2s_0$ to obtain the same result about the multi-point Seshadri constant. This can be seen using 
the revised inequality \eqref{eqn:multi-point-1} (due to $l_B \leq 2s_0$) and the fact that $\mu =2$, the inequality \eqref{inequality-MP-SC} becomes 
\[
\frac{L \cdot C}{\sum\limits_{i=1}^r {\rm mult}_{x_i}C } \geq \frac{a}{2 t_0} + \frac{b}{2 s_0} \geq \min\left\{\frac{a}{t_0} , \frac{b}{s_0}\right\}.
\]
\end{remark}

Now we state and prove similar results on hyperelliptic surfaces of even type. On type 2 surfaces there are $4$ reducible fibres of $\Psi$ and each of them have multiplicity $2$. On such surface we prove the following result.
\begin{proposition}\label{prop:multi-point-even-type-2}
Let $X$ be a hyperelliptic surface of type $2$ and let $L \equiv (a,b)$ be an ample line bundle on $X$. Let $x_1,x_2,\ldots, x_r$ be $r$ points on $X$ lying on singular fibres of $\Psi$ such that $l_B = s_0$ and $l_A = t_0 = 4$.  
Then
\[
\s(X, L, x_1,x_2,\ldots, x_r) = \min\left\{\frac{a}{2} , \frac{b}{s_0} \right\}.
\]
\end{proposition}
\begin{proof}
Since $t_0 = 4$, there is a fibre $(0, 2)$ of $\Phi$ containing $4$ points among $x_1,x_2,\ldots, x_r$ and therefore the Seshadri ratio corresponding to it is 
\[
\frac{L\cdot (0, 2)}{\sum\limits_{i=1}^4 1} = \frac{2a}{4} = \frac{a}{2}. 
\]
The Seshadri ratio corresponding to the fibre of $\Psi$ containing $s$ points among $x_1,x_2,\ldots, x_r$ is
\[
\frac{L\cdot (\mu, 0)}{\sum\limits_{i=1}^4 2} = \frac{2b}{2s_0} = \frac{b}{s_0}.
\]
Therefore, $\s(X, L, x_1,x_2,\ldots, x_r) \leq \min\left\{\frac{a}{2} , \frac{b}{s_0} \right\}$. To prove the other inequality, let $C \equiv (\alpha, \beta)$ ($\alpha, \beta \neq 0$) be a reduced and irreducible curve passing through $x_1,x_2,\ldots, x_r$ with 
multiplicities ${\rm mult}_{x_1}C, {\rm mult}_{x_2}C, \ldots, {\rm mult}_{x_r}C$ respectively. Applying Bézouts theorem to $C$, $(\mu, 0)$ and $C$, $(0, 2)$, and summing over all the respective fibres, we get
\begin{center}
$\mu \beta = C \cdot (\mu, 0) \geq \sum {\rm mult}_{x_i}C \cdot {\rm mult}_{x_i}(\mu, 0) \quad  \Rightarrow \quad 4 \beta \geq  m := \sum\limits_{i=1}^r {\rm mult}_{x_i}$ \\
$2\alpha = C \cdot (0, 2) \geq \sum {\rm mult}_{x_i}C \cdot {\rm mult}_{x_i}(0, 2) \quad  \Rightarrow  \quad  2s_0\alpha \geq m.$
\end{center}
This gives
\[
\frac{L\cdot C}{m} = \frac{a\beta + b\alpha}{m} \geq \frac{a}{4} + \frac{b}{2s_0} = \frac{1}{2}\left(\frac{a}{2} + \frac{b}{s_0} \right) \geq \min\left\{\frac{a}{2} , \frac{b}{s_0} \right\}.
\]
\end{proof}

On hyperelliptic surface of type 4, there are three reducible fibres of $\Psi$ of multiplicities $2, 4$ and $4$. 

\begin{proposition}\label{prop:multi-point-even-type-4}
Let $X$ be a hyperelliptic surface of type $4$ and let $L \equiv (a,b)$ be an ample line bundle on $X$. Let $x_1,x_2,\ldots, x_r$ be $r$ points on singular fibres of $\psi$ such that $l_B = s_0$ and $l_A = t_0 = 2$. 
Then
\[
\s(X, L, x_1,x_2,\ldots, x_r) = \min\left\{a , \frac{b}{s_0}\right\}.
\]
\end{proposition}
\begin{proof}
Fibre of $\Psi$, i.e., the singular fibre containing $s_0$ points among $x_1,x_2,\ldots, x_r$ gives rise to $ \frac{b}{s_0}$ as its Seshadri ratio. 
The fibre of $\Phi$ containing $t_0 = 2$ number of points give rise to $a$ as their Seshadri constant. Therefore,
\[
\s(X, L, x_1,x_2,\ldots, x_r) \leq \min\left\{a , \frac{b}{s_0} \right\}.
\]
To prove the other inequality, let $C \equiv (\alpha, \beta)$ ($\alpha, \beta \neq 0$) be a reduced and irreducible curve passing through $x_1,x_2,\ldots, x_r$ with multiplicities ${\rm mult}_{x_1}C, {\rm mult}_{x_2}C, \ldots, {\rm mult}_{x_r}C$ respectively. 
Applying Bézout's theorem to $C$ with fibres $(\mu, 0)$ and $(0,2)$ we obtain the following:
\begin{eqnarray*}
2\alpha = C\cdot (0,2) \geq \sum {\rm mult}_{x_i}C \cdot {\rm mult}_{x_i}(0, 2) \quad \Rightarrow \quad  2s_0 \alpha  \geq m := \sum\limits_{i =1}^r {\rm mult}_{x_i}C \\
\mu \beta = C \cdot (\mu, 0) \geq \sum {\rm mult}_{x_i}C \cdot {\rm mult}_{x_i}(\mu, 0)  \quad \Rightarrow \quad 2 \mu \beta \geq m := \mu \sum\limits_{i =1}^r {\rm mult}_{x_i}C
\end{eqnarray*}
This further implies
\begin{eqnarray*}
\frac{L\cdot C}{m} = \frac{a\beta + b\alpha}{m} \geq \frac{a}{2} + \frac{b}{2s_0} \geq \min\left\{a , \frac{b}{s_0} \right\}.
\end{eqnarray*}
\end{proof}

On hyperelliptic surface of type $6$ there are three reducible fibres of $\Psi$ each having multiplicities $3$, therefore, $t_0 \leq 3$.
\begin{proposition}\label{prop:multi-point-even-type-6}
Let $X$ be a hyperelliptic surface of type $6$ and let $L \equiv (a,b)$ be an ample line bundle on $X$. Let $x_1,x_2,\ldots, x_r$ be $r$ points on singular fibres of $\psi$ such that $l_B = s_0$ and $l_A = t_0 = 3$. 
Then
\[
\frac{2}{3}\min\left\{a, \frac{b}{s}\right\} \leq \s(X, L, x_1,x_2,\ldots, x_r) \leq \min\left\{a, \frac{b}{s_0}\right\}.
\]
\end{proposition}
\begin{proof}
The singular fibre $(\mu, 0)$ containing $s_0$ number of points give rise to the Seshadri ratio $\frac{3b}{3s_0} = \frac{b}{s_0}$, and the corresponding Seshadri ratio for the curve $(0, 3)$ containing $3$ points among $x_1,x_2,\ldots, x_r$ 
is $a$. Therefore 
\[
\s(X, L, x_1,x_2,\ldots, x_r) \leq \min\left\{a, \frac{b}{s_0}\right\}.
\]
To see the other inequality, let $C \equiv (\alpha, \beta)$ ($\alpha, \beta \neq 0$) be a reduced and irreducible curve passing through the points $x_1,x_2,\ldots, x_r$ with multiplicities ${\rm mult}_{x_1}C, {\rm mult}_{x_2}C, \ldots, {\rm mult}_{x_r}C$ respectively. 
Bézouts theorem applied to $C$ with fibres $(\mu, 0)$ and $(0, 2)$, we get
\begin{eqnarray*}
3\alpha = C\cdot (0, 3) \geq \sum {\rm mult}_{x_i}C \cdot {\rm mult}_{x_i}(0, 3) \quad \Rightarrow \quad  3s \alpha  \geq m := \sum\limits_{i =1}^r {\rm mult}_{x_i}C \\
\mu \beta = C \cdot (\mu, 0) \geq \sum {\rm mult}_{x_i}C \cdot {\rm mult}_{x_i}(\mu, 0)  \quad \Rightarrow \quad 3 \mu \beta \geq m := \mu \sum\limits_{i =1}^r {\rm mult}_{x_i}C
\end{eqnarray*}
This further implies
\begin{eqnarray*}
\frac{L\cdot C}{m} = \frac{a\beta + b\alpha}{m} \geq \frac{a}{3} + \frac{b}{3s_0} = \frac{2}{3}\left( \frac{a}{2} + \frac{b}{2s_0} \right) \geq \frac{2}{3}\min\left\{a , \frac{b}{s_0} \right\}.
\end{eqnarray*}
\end{proof}

\subsection{Seshadri constants on blow-up of hyperelliptic surfaces}

Let $X_r$ be the blow-up of a hyperelliptic surface $X$ along $r$ general points. In this sub-section, we first focus on computing and bounding the Seshadri constant of an ample line bundle 
at a general point of $X_r$, and then on computing the global Seshadri constants. 
\begin{lemma}\label{lemma; for smooth A}
Let $X_r$ denote the blow-up of hyperelliptic surface of odd type (i.e., of type $1,3,5,$ or $7$) at $r$ general points. Let $L \equiv (a, b, \textbf{d})$ be an ample line bundle on $X_r$ such that 
$a, b \geq 2kd$ for some $k \geq \sqrt{r/2}$. Then, for $x \in \textbf{Smooth } A$, we have
\[
\s(X_r, L, x) \leq \begin{cases}
                           \min\{\mu b, a\}, & \text{if } x \not\in \bigcup\limits_{j=1}^r (B-E_j) \\
                           \min\{\mu b, a-d \}, & \text{if } x \in B-E_j
                           \end{cases}
\]
with equality if either $a \leq \frac{\mu}{2\mu -1}b$ or $a \geq \mu(2\mu - 1)b$.
\end{lemma}
\begin{proof}
Since $x$ lies on a smooth fibre isomorphic to $A$, which we denote as $(\mu, 0, 0)$, the Seshadri ratio corresponding to this curve is 
\[
\frac{L \cdot (\mu, 0, 0)}{1} = \mu b.
\]
Now, if $x \in B-E_j$ ($\equiv (0, 1, 1_j)$) be on a strict transform of fibre containing one of the blown-up point $x_j$, the corresponding Seshadri ratio is 
\[
\frac{L \cdot (0, 1, 1_j)}{1} = a - d.
\]
Otherwise, there is always a copy of $B\equiv (0,1,0)$ passing through $x$, and the corresponding Seshadri ratio is $a$. Therefore, the inequality in the statement is clear. 
To see the equality, we consider the two cases: \\
\textbf{Case 1:} $x \not\in \bigcup_{j=1}^r (B-E_j)$\\
Let $C \equiv (\alpha, \beta, \delta_1, \ldots, \delta_r)$ be a curve different from the curves considered above and passing through $x$ with multiplicity $m$. 
Then by B\'ezout's theorem we have $C \cdot (\mu, 0, 0) \geq m$ and $C \cdot (0,1,0) \geq m$. 
This gives $\mu \beta \geq m$ and $\alpha \geq m$, respectively. Using the fact that $a, b \geq 2kd$ for some $k \geq \sqrt{\frac{r}{2}}$, we get
\begin{eqnarray}\label{equality; for seshadri ratio}
\frac{L \cdot C}{m} &=& \frac{a\beta + b \alpha - d\sum \delta_i}{m} \\ \label{inequality; for general x}
                               &\geq& \frac{(a/2)\beta + (b/2) \alpha + d(k, k, \textbf{1})\cdot C}{m} \\ \label{inequality; for smooth x}
                               &\geq& \frac{a}{2\mu} + \frac{b}{2}.
\end{eqnarray}
The last inequality follows since $k$ is chosen so that $(k,k, \textbf{1})$ is a nef line bundle on $X_r$. 

Now, if $a \leq \frac{\mu}{2\mu - 1}b$, then $\min\{\mu b, a\} = a$. Furthermore, following the inequality \ref{inequality; for smooth x}, we note
\begin{eqnarray*}
\frac{a}{2\mu} + \frac{b}{2} \geq a \quad \Leftrightarrow \quad \frac{b}{2} \geq \frac{2\mu - 1}{2\mu}a \quad \Leftrightarrow \quad b \ge \frac{2\mu - 1}{\mu}a.
\end{eqnarray*}
Therefore, we obtain $\varepsilon(X_r, L, x) = a$. 

On the other hand, if $a \geq \mu(2\mu - 1)b$ then $\min\{\mu b, a\} = \mu b$ and agin following the inequality \ref{inequality; for smooth x} we have
\begin{eqnarray*}
\frac{L \cdot C}{m}  \geq \frac{a}{2 \mu} + \frac{b}{2} \geq \mu b &\Leftrightarrow& a + \mu b \geq 2\mu^2 b \\
                                                                                        &\Leftrightarrow& a \ge \mu(2\mu - 1)b,
\end{eqnarray*}
which holds by assumption. Therefore, we obtain $\varepsilon(X_r, L, x) = \mu b$. \\
\textbf{Case 2:} $x \in B-E_j$\\
As in the case 1, using B\'ezout's theorem, we have $C \cdot (\mu, 0, 0) \geq m$ and $C \cdot (0,1, 1_j) \geq m$. These gives $\mu \beta \geq m$ and $\alpha \geq m + d$, respectively.
Now, if $a \leq \frac{\mu}{2 \mu -1}b$, then $\min\{\mu b, a - d \} = a-d$. Continuing from inequality \ref{inequality; for smooth x}, we see that 
\begin{eqnarray}\nonumber
\frac{L \cdot C}{m}  \geq \frac{a}{2 \mu} + \frac{b}{2} \geq a - d &\Leftrightarrow& \frac{b}{2} + d \geq \frac{2 \mu - 1}{2 \mu}a  \\ \nonumber
                                                                                        &\Leftrightarrow& b + 2d \ge \frac{2 \mu - 1}{\mu}a \\ \label{eqn; smooth A and min is a-d}
                                                                                        &\Leftrightarrow& \frac{\mu}{2 \mu - 1}b + \frac{2 \mu}{2 \mu - 1}d \geq a,
\end{eqnarray}
which holds if $a \leq \frac{\mu}{2 \mu - 1}b$. Therefore, we get $\varepsilon(X_r, L, x) = a-d$.

On the other hand, if $a \ge \mu(2\mu - 1)b$, we have $a - d \ge \mu(2\mu - 1)b - d > \mu b$, and as in case 1, we get $\frac{L\cdot C}{m} \geq \mu b$. 
This completes the proof of the lemma. 
\end{proof}

\begin{lemma}\label{lemma; for smooth A-Ej}
Let $X_r$ denote the blow-up of hyperelliptic surface of odd type at $r$ general points. Let $L \equiv (a, b, \textbf{d})$ be an ample line bundle on $X_r$ such that 
$a, b \geq 2kd$ for $k \geq \sqrt{r/2}$. Then, for $x \in A-E_i$, we have
\[
\s(X_r, L, x) \leq \begin{cases}
                           \min\{\mu b - d, a\}, & \text{if } x \not\in \bigcup\limits_j (B-E_j) \\
                           \min\{\mu b - d, a-d \}, & \text{if } x \in B-E_j
                           \end{cases}
\]
with equality if either $a \leq \frac{\mu}{2\mu - 1}b$ or $a \geq \mu(2\mu - 1)b - 2\mu d$.
\end{lemma}

\begin{proof}
Since $x \in A - E_i \equiv (\mu, 0, 1_i)$, which is smooth, the Seshadri ratio corresponding to it is $\mu b - d$. As noted in the proof of Lemma \ref{lemma; for smooth A}, 
$x$ either lies on a copy of $B$ or on $B-E_j$ for some $j$, and their corresponding Seshadri ratio's are $a$ and $a-d$, respectively. Therefore, the inequality in the lemma 
is clear. To adress the remaining part, we proceed as in Lemma \ref{lemma; for smooth A}. 

Let $C \equiv (\alpha, \beta, \delta_1, \ldots, \delta_r)$ be a curve, different from the strict transforms of fibres of $\Phi$ and $\Psi$, passing through $x$ with multiplicity $m$. 
By B\'ezout's theorem, we get $C \cdot (\mu, 0, 1_i) \geq m$ and $C \cdot (0,1,0) \geq m$ (or $C \cdot (0,1,1_j) \geq m$), 
which gives $\mu\beta \ge m + d$ and $\alpha \ge m$ (or $\alpha \ge m + d$), respectively. Therefore, as in Lemma \ref{lemma; for smooth A}, we have the inequality \ref{inequality; for smooth x}. 
We now consider the following two cases. 

\textbf{Case 1:} $x \not\in \bigcup_{j=1}^r (B-E_j)$ \\
Note that if $a \leq \frac{\mu}{2\mu -1}b$, then $a + d < 2a < (2\mu -1)a \leq \mu b$. Therefore, $\min\{\mu b - d, a\} = a$.   
Now, following the inequality \ref{inequality; for smooth x} and proceeding as in the Case 1 of Lemma \ref{lemma; for smooth A}, we note:
\begin{eqnarray*}
\frac{L \cdot C}{m}  \geq \frac{a}{2 \mu} + \frac{b}{2} \geq a &\Leftrightarrow& \frac{b}{2}  \geq \frac{2\mu - 1}{2\mu}a \\                                                                                    
                                                                                        &\Leftrightarrow& \frac{\mu}{2\mu - 1}b \geq a.
\end{eqnarray*}
On the other hand, if $a \ge \mu(2\mu -1)b - 2\mu d$, then $\min\{\mu b - d, a\} = \mu b - d$ and 
\begin{eqnarray*}
\frac{L \cdot C}{m}  \geq \frac{a}{2\mu} + \frac{b}{2} \geq \mu b - d &\Leftrightarrow& a  \geq \mu(2\mu - 1)b - 2\mu d.   \label{eqn; smooth A-Ei and min is 2b - d}                                                                                 
%
\end{eqnarray*}
Therefore, we get the equality as stated i.e., $\varepsilon(X_r, L,x) = \min\{\mu b - d, a\}$.

\textbf{Case 2:} $x \in B-E_j$ \\
In this case, we see that $\min\{\mu b - d, a - d\} = a - d$ if $a \le \frac{\mu}{2\mu - 1}b$. Therefore, $\frac{L \cdot C}{m} \geq a - d$ follows from 
the characterisation in \ref{eqn; smooth A and min is a-d} of Lemma \ref{lemma; for smooth A}. On the other hand, if $a \ge \mu(2\mu - 1)b - 2\mu d$, we get $\min\{\mu b - d, a - d\} 
= \mu b - d$, and as seen in the characterisation \ref{eqn; smooth A-Ei and min is 2b - d} of the previous lemma, we have $\frac{L \cdot C}{m} \geq \mu b - d$. 
Thus, we get the equality in this case too, i.e., $\varepsilon(X_r, L, x) = \min\{\mu b - d, a-d\}$.
\end{proof}

\begin{lemma}\label{lemma; for singular A}
Let $X_r$ denote the blow-up of hyperelliptic surface of odd type at $r$ general points. Let $L \equiv (a, b, \textbf{d})$ be an ample line bundle on $X_r$ such that 
$a, b \geq 2kd$ for $k \geq \sqrt{r/2}$. Then, for $x \in \textbf{Singular } A$, we have
\[
\s(X_r, L, x) = \begin{cases}
                           \min\{b, a\} ,& \text{if } x \not\in \bigcup_{j=1}^r (B-E_j) \\
                           \min\{a-d, b\}, & \text{if } x \in B-E_j.
                           \end{cases}
\]
\end{lemma}
\begin{proof}
Since $x \in \textbf{Singular } A$ and, by definition, $\textbf{Singular } A$ represents a singular fibre of $\Psi$ with the highest multiplicity $\mu$, the corresponding Seshadri ratio 
is $b$. Depending on whether the point $x$ lies on $B-E_j$ for some $j$ or directly on $B$, the corresponding Seshadri ratio is either $a-d$ or $a$, 
respectively. Consequently, $\varepsilon(X_r, L,x)$ is bounded above by $\min\{a-d, b\}$ in the former case and $\min\{a, b\}$ in the latter. 

To see the equality, let $C \equiv (\alpha, \beta, \delta_1, \ldots, \delta_r)$ be a reduced and irreducible curve (different from the strict transforms of fibres of $\Phi$ and $\Psi$) passing through $x$ with multiplicity $m$. 
By B\'ezout's theorem, we have $C \cdot (\mu, 0, 0) \geq \mu m$ and $C \cdot (0,1,0) \geq m$ or $C \cdot (0,1,1_j) \geq m$. This gives $\beta \ge m$ and $\alpha \geq m$ or $\alpha -d \geq m$, respectively. 
Therefore, the inequality \ref{inequality; for general x} leads to the following
\begin{eqnarray}\label{inequality; for singular x}
\frac{L \cdot C}{m} &\geq& \frac{a}{2} + \frac{b}{2} \\ \nonumber
                              &\geq& \min\{a, b\}.
\end{eqnarray}
Thus, if $x \not\in \bigcup_{j=1}^r (B-E_j)$, we get $\varepsilon(X_r, L, x) = \min\{a, b\}$. \\
Now, if $x \in B-E_j$ for some $j$, and if $\min\{a-d, b\} = b$, then we have 
\[
\frac{a}{2} - \frac{d}{2} \geq \frac{b}{2} \Rightarrow \frac{a}{2} + \frac{b}{2} \geq b + \frac{d}{2} > b.
\]
Therefore, using inequality \ref{inequality; for singular x}, we get the equality i.e., $\varepsilon(X_r, L,x) = b$. On the other hand if $\min\{a-d, b\} = a - d$, using 
inequality \ref{inequality; for singular x}, we again note that 
\[
\frac{L \cdot C}{m} \ge \frac{a}{2} + \frac{b}{2} > \frac{a-d}{2} + \frac{a-d}{2} = a - d.
\]
Hence $\varepsilon(X_r,L,x) = \min\{a-d, b\}$. This completes the proof of the lemma. 
\end{proof}

We now use the above lemmas to obtain the following.
\begin{theorem}\label{Thm-type1-general-point}
	Let $X$ be a hyperelliptic surface of odd type, and let $\pi_r : X_r \longrightarrow X$ denote the blow-up of $X$ at $r$ general points with exceptional divisors $E_i$.
	Let $L = (a, b, \textbf{d})$ be an ample line bundle on $X_r$ satisfying either $a \geq \mu b$ or $\frac{2\mu - 1}{\mu}a < b$ and $\frac{a}{d}, \frac{b}{d} \geq 2k$, where $k \geq \sqrt{\frac{r}{2}}$ be 
	a positive integer. Let $x \in X_r$ be a very general point. Then, we have 
	\[
	\min\{a,b\} \leq \varepsilon(X_r, L,  x) \leq \min\{a,\mu b\}.
	\]
\end{theorem}
\begin{proof}
Since $x$ is a very general point of $X_r$, we may assume that $x$ lies on a smooth fibre $A$ and not on the fibre $A-E_j$ for any $j$. From Lemma \ref{lemma; for smooth A}, 
$\varepsilon(L, X_r, x) \leq \min\{a,\mu b\}$ follows. Therefore, it only remain to show the lower bound. 

To see this, let $C \equiv (\alpha, \beta, \delta_1, \ldots, \delta_r)$ be a reduced and irreducible 
curve (different from the strict transforms of fibres of $\Phi$ and $\Psi$) passing through $x$ with multiplicity $m$. Using the B\'ezout's theorem, we see that 
\begin{eqnarray*}
C \cdot (\mu, 0, 0) \geq m  \quad \Rightarrow \quad \mu \beta &\geq& m \quad \Rightarrow \quad \frac{a \beta}{2} \geq \frac{a}{2\mu}m.
\end{eqnarray*}
Continuing the calculation from inequality \ref{inequality; for smooth x}, we get
\begin{eqnarray*}
\frac{L\cdot C}{m} \geq \frac{a}{2\mu} + \frac{b}{2} \geq \min\{a, b\}.
\end{eqnarray*}
Where the last inequality follows from the assumption $a \geq \mu b$ or $\frac{2\mu - 1}{\mu}a \leq b$.
\end{proof}

\begin{corollary}\label{Cor:general-point}
Continuing with the notation of the previous theorem, let $L = (a, b, \textbf{d})$ be an ample line bundle on $X_r$ such 
that $a/d, b/d \geq 2k$, where $k \geq \sqrt{r/2}$ be a positive integer. Then 
\begin{itemize}
\item[(1)] $\s(X_r, L, 1) = a$, if $a \leq \frac{\mu}{2\mu - 1}b$\\
\item[(2)] $\s(X_r, L, 1) \in [b, \mu b]$, if $a \geq \mu b$.
\end{itemize}
\end{corollary}
\begin{proof}
If $a \leq \frac{\mu}{2\mu -1}b$, we have $\min\{a,b\} = a = \min\{a, \mu b\}$. Moreover, if $a \geq \mu b$ then $\min\{a,b\} = b$ and $\min\{a, \mu b\}  = \mu b$. Now, the result follows from the previous theorem. 
\end{proof}

\subsection*{Results about  $\s(X,L)$}
Let $X$ be a smooth projective surface, and let $L$ be an ample line bundle on $X$. 
The \textit{global Seshadri constant} $\s(X,L)$ was defined in the introduction. 
We now establish results about rationality of these constants. 

\begin{theorem}\label{Global:SC:Odd-Type}
	Let $\pi_r : X_r \longrightarrow X$ denotes the blow-up of a hyperelliptic surface $X$ of odd type at $r$ general points. Let $L \equiv (a, b, \textbf{d})$ be an ample line bundle on $X_r$ 
	such that $a \geq (\sqrt{r+1})d$ and $b > \frac{r}{\sqrt{r+1} + 1}d$. Then $\s(X_r, L) \in \mathbb{Q}$. 
\end{theorem}

\begin{proof}
Using \cite[Proposition 1.1]{BS1}, it is enough to show the existence of a curve $C \subset X_r$ and a point $x \in C$ such that the corresponding Seshadri ratio 
$\frac{L \cdot C}{m}$ is less or equal to $\sqrt{L^2}$, where $m$ is the multiplicity of $C$ at $x$. 

Note that if $x$ lies on a singular fibre $A \equiv (\mu, 0, 0)$, the 
corresponding Seshadri ratio is $b$. Additionally, if $x \in B-E_j$, then the corresponding Seshadri ratio is $a - d$. We now claim the following: \\
\textbf{Claim:} $\min\{a-d, b\} < \sqrt{L^2}$ \\
If $a-d < b$ then, 
\begin{eqnarray*}
L^2 = 2ab - rd^2 &>& 2a(a-d) - rd^2 \\
&=& (a-d)^2 + a^2 - (r+1)d^2 \\  
&\geq& (a-d)^2,
\end{eqnarray*}
where the last inequality follows since $a \geq (\sqrt{r+1})d$.\\
 We now assume that $a -d \geq b$ and note 
\begin{eqnarray*}
L^2 = 2ab - rd^2 &=& ab + b(a-d) + bd - rd^2 \\
&\geq& b^2 + b(a+d) -rd^2 \\
&\geq& b^2 + b(\sqrt{r+1} + 1)d - rd^2 \\
&>& b^2,
\end{eqnarray*}
since $b > \frac{r}{\sqrt{r+1} + 1}d$ by assumption. 
\end{proof}

In the next theorem, under stricter conditions on $a$ and $b$ compared to the previous theorem, we explicitly determine the global Seshadri constant.
\begin{theorem}\label{Global:SC:Computation}
Let $\pi_r : X_r \longrightarrow X$ denote the blow-up of a hyperelliptic surface $X$ of odd type at $r$ general points. Let $L \equiv (a, b, \textbf{d})$ be an ample line bundle on $X_r$ 
such that $a \notin [\frac{\mu}{2\mu - 1}b + \frac{2\mu}{2\mu - 1}d, \mu b]$, and $a, b \geq 2kd$ for some $k \geq \sqrt{\frac{r}{2}}$. Then  
\begin{eqnarray*}
  \s(X_r, L) = \min\{a -d, b \}.
  \end{eqnarray*}
\end{theorem}

\begin{proof}
For $x \in \textbf{Singular } A$, it is clear from Lemma \ref{lemma; for singular A} that $\s(X_r, L, x) \geq \min\{a -d, b \}$, with equality if $x \in (B-E_j)$ for some $j$. 
Therefore, we assume that $x \in \textbf{Smooth } A$. Let $C \subset X_r$ be a curve passing through $x$ with multiplicity $m$. 
If $x \in \bigcup_{j=1}^r (B-E_j)$, then following the inequality \ref{inequality; for smooth x}, we note that for $C$ (different from $\textbf{Smooth } A$ and $B-E_j$, which gives 
Seshadri ratios as $b$ and $a-d$, respectively), we have
\[
\frac{L \cdot C}{m} \geq \frac{a}{2\mu} + \frac{b}{2} \geq b,
\] 
 if $a \geq \mu b$. Furthermore, if $a \geq \mu b$, then $\min\{a-d, b \} = b$.  

On the other hand, if $ a \leq \frac{\mu}{2\mu - 1}b + \frac{2\mu}{2\mu - 1}d$, then $a -d \leq  \frac{\mu}{2\mu - 1}b + \frac{1}{2\mu - 1}d < b$, as $b > d$. 
Therefore, $\min\{a-d, b \} = a-d$, and hence, again continuing from inequality \ref{inequality; for smooth x}, we see that 
\[
\frac{L \cdot C}{m} \geq \frac{a}{2\mu} + \frac{b}{2} \geq a - d \Leftrightarrow a \leq  \frac{\mu}{2\mu - 1}b + \frac{2\mu}{2\mu - 1}d,
\] 
which holds by assumption. 

Lastly, if $x \in A-E_j$, then the Seshadri ratio corresponding to this curve is $\mu b - d$, which is greater than $b$. Also, since $A - E_j$ is smooth, by the previous discussion 
$\frac{L \cdot C}{m} \ge \min\{a-d, b \}$ for any curve $C$, different from the strict transform of fibre passing through $x$ with multiplicity $m$. Therefore, the minimum value 
is $\min\{a-d, b \}$, and it is attained at $x\in (\textbf{Singular } A) \cap \bigcup_j (B - E_j)$. This completes the proof.
\end{proof}

\begin{example}
Let $X$ be a hyperelliptic surface of type 1 and let $L \equiv (a, b, {\bf 1})$ be a line bundle on $X_8$, with $a = 4, b = 6$. Then $\s(X_8, L, x) = 4$ for a very general point $x \in X$. Since $r = 8$, taking $k = 2$ we see that 
$\frac{a}{d}, \frac{b}{d} \geq 2k$ is satisfied and therefore using the fact that $a = 4 \leq 4 = \frac{2}{3}b = \frac{\mu}{2\mu -1}b $ and Corollary \ref{Cor:general-point}, we get $$\s(X_8, L, 1) = 4.$$

Note also that since $4 \notin [12, 12] = [\frac{2}{3}6 + \frac{4}{3}6, 12]$, it follows from Theorem \ref{Global:SC:Computation} that $$\s(X_8, L) = \min\{3, 6 \} = 3.$$
\end{example}

\begin{theorem}\label{Global:SC:Even-Type-2&4}
	Let $X$ be a hyperelliptic surface of type 2 or 4, and let $\pi_r : X_r \longrightarrow X$ be a blow-up of $X$ at $r$ general points. Let $L \equiv (a, b, \textbf{d})$ be an ample line bundle on $X_r$ 
	such that $b > rd$. Then $\s(X_r, L) \in \mathbb{Q}$.
\end{theorem}

\begin{proof}
	By \cite[Proposition 1.1]{BS1}, it is enough to exhibit a curve $C \subset X_r$ passing through a point $x \in X_r$ such that $(L \cdot C)/\mult_xC < \sqrt{L^2}$. 
	On each hyperelliptic surface of even type, we will exhibit such curve.\\
	\textbf{Type 2 and 4:} From \cite[Theorem 1.4]{Se}, the generators for ${\rm Num}(X)$ are $(2,0)$, $(0,2)$ for type 2 and $(4,0)$, $(0,2)$ for type 4. Let $x$ 
	be the point of intersection of $(e,0,0)$ and $(0,2,1_i)$ where $(e,0,0)$ corresponds to the singular fibre with highest multiplicity i.e., $e = 2$ on type 2 and 
	$e = 4$ on type 4. The Seshadri ratio corresponding to these curves are 
	\renewcommand\labelitemi{\tiny$\bullet$}
	\begin{itemize}
	\item On type 2:
	\end{itemize}
	\[
	\frac{L \cdot (2,0,0)}{2} = b, \;\; {\rm and} \; \;\frac{L \cdot (0,2,1_i)}{1} = 2a - d,
	\] 
	\renewcommand\labelitemi{\tiny$\bullet$}
	\begin{itemize}
	\item On type 4:
	\end{itemize}
	\[
	\frac{L \cdot (4,0,0)}{4} = b, \;\; {\rm and} \; \; \frac{L \cdot (0,2,1_i)}{1} = 2a - d.
	\] 
	We now prove that $\min\{2a-d, b\} < \sqrt{L^2}$. This follows, if $2a-d < b$, since
	\begin{eqnarray*}
		(2a-d)^2  &<& (2a-d)b  \\
		&<&  L^2 = 2ab - rd^2 \\
		\Leftrightarrow rd &<& b.
	\end{eqnarray*}
	On the other hand, if  $2a-d \geq b$, then multiplying both sides by $b$, we get $2ab - bd \geq b^2$, and as in the previous case, $L^2 > 2ab - bd$ if and only if $b > rd$.
	\end{proof}

\begin{theorem}\label{Global:SC:Even-Type-6}
	Let $X$ be a hyperelliptic surface of type 6, and let $\pi_r : X_r \longrightarrow X$ be a blow-up of $X$ at $r \geq 3$ general points. Let $L \equiv (a, b, \textbf{d})$ be an ample line bundle on $X_r$ 
	such that $a > \frac{r+1}{2}d$, $b > rd$. If either $b \geq \frac{9}{2}a - 2d$ or $b \leq 2(a - d)$, then $\s(X_r, L) \in \mathbb{Q}$.
\end{theorem}
\begin{proof}
	On hyperelliptic surface of type 6, the generators for ${\rm Num}(X)$ are $(3,0)$ and $(0,3)$. Therefore, if we take $x$ to be the point of intersection of $(3,0,0)$ and $(0,3,1_i)$, the Seshadri ratio of these two curves are 
	\[
	\frac{L \cdot (3,0,0)}{3} = b, \;\; {\rm and} \;\;	\frac{L \cdot (0,3,1_i)}{1} = 3a - d.
	\] 
	Now, if $b \geq \frac{9}{2}a - 2d$, then
	\[
	3a - d + (\frac{3}{2}a - d) \geq 3a-d \quad \Leftrightarrow \quad a \geq \frac{4}{3}d,
	\]
	which holds since $r \geq 3$. Therefore, we obtain $\min\{3a-d, b \} = 3a - d$, and in that case, we must show the following
	\begin{eqnarray*}
	(3a - d)^2 < L^2 = 2ab - rd^2 \\
	\Leftrightarrow 9a^2 + (r+1)d^2 < 2ab + 6ad.
	\end{eqnarray*}
	However, since
	\[
	2ab + 6ad > 2a(\frac{9}{2}a - 2d) + 6ad = 9a^2 + 2ad > 9a^2 + (r+1)d^2,
	\]
	the inequality follows. \\
	On the other hand, if $b \leq 2(a - d)$, we find that
	\[
	L^2 = 2ab - rd^2 \geq (b+d)b - rd^2 = b^2  + bd - rd^2 > b^2,
	\]
	since $b > rd$. Therefore, $\min\{3a-d, b \} < L^2$ and the theorem follows from \cite[Proposition 1.1]{BS1}.
\end{proof}

\begin{example}
Let $X_5$ denote the blow up of a hyperelliptic surface of type 2, 4, or 6. Let $L \equiv (4, 6, {\bf 1})$ be a line bundle on $X_5$. Then, we have $a = 4, b = 6, d= 1$ and $r=5$. Since 
\[
4 > 3 = \frac{r+1}{2}d, \quad 6 > 5 = rd \quad \text{and} \quad 6 \leq 6 = 2(a - d),
\] 
if follows from the Theorem \ref{Global:SC:Even-Type-2&4} and Theorem \ref{Global:SC:Even-Type-6} that $\s(X_5, L) \in \mathbb{Q}$. \\
Note that the same conclusion holds for any $a$ such that $L$ is ample on surfaces of type 2 and 4. 
\end{example}

\section*{Acknowledgement}
The author would like to thank \L{}ucja Farnik for pointing out the reference to Oliver K\"{u}chle \cite{Kuchle}.

\bibliographystyle{plain}

\end{document}